\documentclass[11pt]{article}

\usepackage[letterpaper,top=2cm,bottom=2cm,left=2cm,right=2cm,marginparwidth=1.75cm]{geometry}

\usepackage{amsmath}
\usepackage{amsthm}
\usepackage{amssymb}
\usepackage{graphicx}
\usepackage{nccmath}
\usepackage{setspace}
\usepackage{enumitem}
\usepackage{float}
\usepackage{bm}
\usepackage[colorlinks=true, allcolors=blue]{hyperref}

\newtheorem{Theorem}{Theorem}[section]
\newtheorem{Lemma}[Theorem]{Lemma}
\newtheorem{Corollary}[Theorem]{Corollary}

\theoremstyle{definition}

\newtheorem{Example}[Theorem]{Example}
\newtheorem{Remark}[Theorem]{Remark}

\renewcommand{\leq}{\leqslant}
\renewcommand{\geq}{\geqslant}

\newcommand{\B}{{\ensuremath{\mathcal{B}}}}

\newcommand{\E}{\mathcal{E}}

\def \mod{\pmod}

\let\oldproofname=\proofname
\renewcommand{\proofname}{\rm\bf{\oldproofname}}

\title{Bounds on data limits for all-to-all comparison from combinatorial designs}
\author{Joanne Hall\thanks{School of Science, RMIT University, Melbourne VIC 3001, Australia} \qquad Daniel Horsley\thanks{School of Mathematics, Monash University, Clayton VIC 3800, Australia} \qquad Douglas R. Stinson\thanks{David R. Cheriton School of Computer Science, University of Waterloo, Waterloo ON, N2L 3G1, Canada}}
\date{}

\begin{document}
\setstretch{1.1}
\maketitle

\begin{abstract}
In situations where every item in a data set must be compared with every other item in the set, it may be desirable to store the data across a number of machines in such a way that any two data items are stored together on at least one machine. One way to evaluate the efficiency of such a distribution is by the largest fraction of the data it requires to be allocated to any one machine. The \emph{all-to-all comparison (ATAC) data limit for $m$ machines} is a measure of the minimum of this value across all possible such distributions. In this paper we further the study of ATAC data limits. We observe relationships between them and the previously studied combinatorial parameters of \emph{fractional matching numbers} and \emph{covering numbers}. We also prove a lower bound on the ATAC data limit that improves on one of Hall, Kelly and Tian, and examine the special cases where equality in this bound is possible. Finally, we investigate the data limits achievable using various classes of combinatorial designs. In particular, we examine the cases of transversal designs and projective Hjelmslev planes.
\end{abstract}

\section{Introduction}

A \emph{hypergraph} is a pair $(X,\E)$ where $X$ is a finite set and $\E$ is a finite multiset of subsets of $X$. Elements of $X$ are called \emph{vertices} and elements of $\E$ are called \emph{edges}. If $|E|=k$ for each $E \in \E$, then the hypergraph is said to be \emph{$k$-uniform}. A \emph{covering design} is a hypergraph $(X,\E)$ with the property that each pair of elements of $X$ occur together in at least one edge in $\E$. Sometimes the vertices of a covering design are referred to as \emph{points} and its edges as \emph{blocks}. A \emph{linear space} is a covering design in which each pair of points occur together in exactly one block. For a positive integer $s$, a \emph{projective plane of order $s$} is an $(s+1)$-uniform linear space with $s^2+s+1$ points and $s^2+s+1$ blocks. For an integer $s \geq 2$, an \emph{affine plane of order $s$} is an $s$-uniform linear space with $s^2$ points and $s^2+s$ blocks. It is known that, for each order $s \geq 2$, either both of these planes exist or neither do, and further that both exist if $s$ is a power of a prime (see \cite[\S2.3]{Sti}, for example).  An \emph{almost projective plane of order $s$} is an $(s+1)$-uniform covering design with $s^2+s$ points and $s^2+s$ blocks in which each point is in exactly $s+1$ blocks. Almost projective planes of orders $2$ and $3$ exist but no example of order greater than 3 has been found (see \cite{BieMarPam,BloJunSch}). A \emph{weighting} of a set $S$ is an assignment of nonnegative real values to the elements of $S$. For such a weighting $w$ we abbreviate $\sum_{s \in S'}w(s)$ to $w(S')$ for each subset $S'$ of $S$. We say a such a weighting $w$ is \emph{normalised} if $w(S)=1$.

For a covering design $D=(X,\B)$ and a normalised weighting $w$ of $X$ we define $L(D,w)$ to be $\max\{w(B):B \in \B\}$. We then define $L(D)$ to be the infimum of $L(D,w)$ over all normalised weightings $w$ of $X$. Following Hall, Kelly and Tian \cite{HalKelTia}, for each positive integer $m$, we define the \emph{ATAC data limit for $m$ machines}, denoted $L(m)$, to be the infimum of $L(D)$ over all covering designs $D$ with $m$ blocks. We will see in Lemma~\ref{L:LVsFracIndepSet} that these infimums are in fact minimums. We discuss the motivating problem and the reason for these definitions in Section~\ref{S:problem}. For any covering design $D$ we have $L(D')=L(D)$ where $D'$ is a covering design obtained from $D$ by adding a new block that is a subset of an existing block. So $L(m)$ is nonincreasing in $m$. We summarise some of the main results of \cite{HalKelTia} in the following theorem.

\begin{Theorem}[\cite{HalKelTia}]\label{T:HKTBound}
For each positive integer $m$,
\begin{equation}\label{E:HKTBound}
L(m) \geq \min \biggl\{ \mfrac{1}{\lfloor \sqrt{m} \rfloor}  , \mfrac{\lfloor \sqrt{m} \rfloor + 1}{m} \biggr\} .
\end{equation}
Furthermore, we have equality in \eqref{E:HKTBound} if $m=s^2+s+1$ and there exists a projective plane of order $s$ or if $m=s^2+s$ and there exists an affine plane of order $s$.
\end{Theorem}

In this paper we will further the investigation initiated in \cite{HalKelTia} a number of ways. In Section~\ref{S:relat}, we observe some connections between the data limit and previously studied parameters, namely covering numbers and fractional matching numbers. In Section~\ref{S:projAff}  we investigate the cases where $m \in \{s^2+s,s^2+s+1\}$ for some positive integer $s$ and investigate when equality in \eqref{E:HKTBound} is possible. We prove the following theorem and, in addition,
also characterise the covering designs that can achieve these bounds (see Lemmas~\ref{L:projPlanesEtc} and \ref{L:affinePlanesEtc} for details).

\begin{Theorem}\label{T:projAff} \phantom{a}
\begin{itemize}
    \item[\textup{(a)}]
For each positive integer $s$, we have $L(s^2+s+1)=\frac{s+1}{s^2+s+1}$ if and only if a projective plane of order $s$ exists. Furthermore, if $D$ is a covering design with $s^2+s+1$ blocks and at most $s^2+s+1$ points such that $L(D)=\frac{s+1}{s^2+s+1}$, then $D$ is a projective plane of order $s$.
    \item[\textup{(b)}]
For each integer $s \geq 2$, we have $L(s^2+s)=\frac{1}{s}$ if and only if either an affine plane of order $s$ exists or an almost projective plane of order $s$ exists. If $D$ is a covering design with $s^2+s$ blocks and at most $s^2$ points such that $L(D)=\frac{1}{s}$, then $D$ is an affine plane of order $s$.
\end{itemize}
\end{Theorem}

In Section~\ref{S:newCBound}, we note that Theorem~\ref{T:HKTBound} establishes the asymptotics for the data limit and, further, we provide a stronger lower bound on the data limit.

\begin{Theorem}\label{T:newBound}
For each integer $m \geq 2$, we have
\begin{equation}\label{E:newBound}
L(m) \geq \mfrac{s^2 + (2m-1)s - 1  + \sqrt{(s^2-s-1)^2 +
  4m(s-1)(s^2 + s + 1 - m)}}{2 m (s^2+s-1)}
\end{equation}
where $s$ is the positive integer such that $s^2-s+1 < m \leq s^2+s+1$. Furthermore, if we have equality in \eqref{E:newBound}, then $m=4$ or $m \in \{s^2+s,s^2+s+1\}$.
\end{Theorem}

We will see in Lemma~\ref{L:boundFacts} that \eqref{E:newBound} agrees with \eqref{E:HKTBound} when $m \in \{s^2+s,s^2+s+1\}$ and is strictly stronger for all other values of $m$.

In Section~\ref{S:designClasses}, we investigate values of $L(D)$ subject to the requirement that $D$ is a covering design of a certain type and establish data limits for transversal designs and finite projective Hjelmslev planes. In Section~\ref{S:conc} we give some concluding thoughts and discuss some open problems.

\section{Describing the All to All Comparison Problem}\label{S:problem}

In situations where every item in a data set must be compared with every other item in the set, it may be desirable to store the data across a number of machines in such a way that any two data items are stored together on at least one machine \cite{HalKelTia,ZhaTiaFidKel,ZhaTiaKelFid}. When the data set is large and the storage capacity of the machines is limited, we may further want to do this in a way that minimises the fraction of the data that needs to be stored on any given machine. To do this, we partition the data set into a number of \emph{data groups}; each data group is represented by a point of a covering design $D$, and a normalised weighting $w$ of these points indicates the fraction of the data allocated to each group. Each machine is represented by a block of $D$ that contains exactly the data groups stored on that machine. The requirement that any two data items be stored together on some machine corresponds exactly to the requirement that the blocks of $D$ form a covering design. The maximum fraction of the data that needs to be stored on any one machine is then $L(D,w)$. Consequently that data limit $L(m)$ represents the minimum such fraction over all possible groupings of the data and all possible allocations of the groups to $m$ machines.

To illustrate our definitions we determine $L(m)$ for $m \in \{1,2,3,4\}$. In any covering design $(X,\B)$ with $|\B| \in \{1,2\}$, it can be seen that entire point set $X$ must be a block in $\B$. From this it follows that $L(m)=1$ for $m \in \{1,2\}$.

\begin{Example}\label{X:order3}
Let $D=(X,\B)$ be the covering design with point set $X=\{1,2,3\}$ and block set $\B=\{\{1,2\},\{1,3\},\{2,3\}\}$. So $D$ is a projective plane of order 1 (note that we allow projective planes of order 1, even though they are sometimes excluded as trivial). Let $w$ be the normalised weighting of $X$ given by $w(x)=\frac{1}{3}$ for each $x \in \{1,2,3\}$. Then $w(B)=\frac{2}{3}$ for each $B \in \B$ and hence $L(D,w)=\frac{2}{3}$. Furthermore, we have $L(3) \geq \frac{2}{3}$ by Theorem~\ref{T:HKTBound} and so this example shows that $L(3)=\frac{2}{3}$.
\end{Example}

\begin{Example}\label{X:order4}
Let $D=(X,\B)$ be the covering design with point set $X=\{1,2,3,4\}$ and block set $\B=\{\{1,2\},\{1,3\},\{1,4\},\{2,3,4\}\}$. Let $w$ be the normalised weighting of $X$ given by $w(1)=\frac{2}{5}$ and $w(x)=\frac{1}{5}$ for each $x \in \{2,3,4\}$. Then $w(B)=\frac{3}{5}$ for each $B \in \B$ and hence $L(D,w)=\frac{3}{5}$. Furthermore, we have $L(4) \geq \frac{3}{5}$ by Theorem~\ref{T:newBound} and so this example shows that $L(4)=\frac{3}{5}$. Since $L(m)$ is nonincreasing in $m$, this implies that $L(m) \leq \frac{3}{5}$ for each integer $m \geq 4$.
\end{Example}

\section{Relationship to other parameters}\label{S:relat}

In this section we observe a relationship between the data limit and two other previously studied combinatorial parameters.

Let $H=(X,\mathcal{E})$ be a hypergraph. Throughout the paper, for each $x \in X$, we denote $\{E \in \mathcal{E}:x \in E\}$ by $\mathcal{E}_x$ and $|\E_x|$ by $r_x$. In hypergraph terms $r_x$ is the degree of vertex $x$, while in design-theoretic terminology $r_x$ is sometimes called the \emph{replication number} of $x$. The \emph{dual} of a hypergraph  is the hypergraph whose vertex set is $\mathcal{E}$ and whose edge set is $\{\mathcal{E}_x:x \in X\}$. Thus the dual is the hypergraph obtained from $H$ by interchanging the role of vertices and edges. We say a hypergraph is \emph{simple} if no two of its edges are incident with exactly the same set of vertices (note that if a hypergraph is not simple then its dual will contain multiple vertices incident with exactly the same set of edges). We say that $H$ is \emph{intersecting} if $E_1 \cap E_2 \neq \emptyset$ for all $E_1,E_2 \in \mathcal{E}$. Observe that a hypergraph is intersecting if and only if its dual is a covering design. We say a weighting $w$ is \emph{rational} if it assigns only rational weights.

A \emph{fractional matching} in a hypergraph $H=(X,\E)$ is a weighting $w$ of $\E$ such that $\sum_{E \in \E_x}w(E) \leq 1$ for each $x \in X$. The \emph{size} of $w$ is $\sum_{E \in \E}w(E)$. The \emph{fractional matching number} $\nu^*(H)$ of $H$ is the maximum size of a fractional matching in $H$. Because $\nu^*(H)$ is the optimal value of a linear program, it is clearly well-defined for any hypergraph $H$. In \cite{Fur}, F{\"u}redi studied fractional matching numbers of intersecting hypergraphs. We first exhibit the relationship between these and data limits.\pagebreak

\begin{Lemma}\label{L:LVsFracIndepSet}\phantom{a}
\begin{itemize}
    \item[\textup{(a)}]
For any covering design $D$, we have that $L(D)=1/\nu^*(H)$ where $H$ is the dual of $D$. Furthermore, $L(D)=L(D,w_0)$ for some rational normalised weighting $w_0$ of the points of $D$.
    \item[\textup{(b)}]
For any positive integer $m$, $L(m)=1/\nu^*(m)$ where $\nu^*(m)$ is the supremum of $\nu^*(H)$ over all intersecting hypergraphs $H$ with $m$ vertices. Furthermore, $L(D)=L(D,w_0)$ for some covering design $D$ with $m$ blocks and some rational normalised weighting $w_0$ of the points of $D$.
\end{itemize}
\end{Lemma}

\begin{proof}
We first prove (a). Say $D=(X,\B)$ and let $H=(\B,\{\B_x : x \in X\})$ be the dual of $D$. Note that $H$ is intersecting. Because $\nu^*(H)$ is the optimal value of a linear program with integer coefficients, there is a fractional matching of size $\nu^*(H)$ in $H$ in which each edge is assigned a rational weight. This induces a rational weighting $w^*$ of the point set of $D$ such that $w^*(B) \leq 1$ for each $B \in \B$. Then, for the rational normalised weighting $w_0$ of $X$ given by $w_0(x)=w^*(x)/\nu^*(H)$ we have $L(D,w_0) \leq 1/\nu^*(H)$. Thus $L(D) \leq 1/\nu^*(H)$.

Now, consider a normalised weighting $w$ of $X$ such that $w(B) \leq L(D)$ for each $B \in \B$. There is a fractional matching $w^*$ of $H$ of size $1/L(D)$ defined by $w^*(\B_x)=w(x)/L(D)$. So $\nu^*(H)\geq 1/L(D)$ or, equivalently, $L(D) \geq 1/\nu^*(H)$. Thus, in view of the preceding paragraph, we have that $L(D)=1/\nu^*(H)$ for any covering design $D$, and that $L(D)=L(D,w_0)$ for some rational normalised weighting $w$ of the points of $D$. So we have proved (a).

The first part of (b) follows directly from (a) and the definition of $L(m)$. There must be a hypergraph $H_0$ with $\nu^*(H_0)=\nu^*(m)$ because there are only finitely many nonisomorphic simple hypergraphs $H$ with $m$ vertices and it suffices to consider only these in the definition of $\nu^*(m)$. Applying (a) to the dual of $H_0$ completes the proof of (b).
\end{proof}

Note that it follows from Lemma~\ref{L:LVsFracIndepSet} that the infimums in the definitions of $L(D)$ and $L(m)$ are in fact minimums. Using the fact that the fractional matching number of a hypergraph is the optimal value of a linear program, F{\"u}redi \cite{Fur} observed the following.

\begin{Lemma}[\cite{Fur}]\label{L:symCover}
Let $H=(X,\mathcal{E})$ be an intersecting hypergraph. Then, there is a subset $\E'$ of $\E$ such that $|\E'| \leq |X|$ and $\nu^*(H')=\nu^*(H)$, where $H'$ is the intersecting hypergraph $(X,\E')$.
\end{Lemma}

This allows us to give an equivalent definition of the ATAC data limit which is sometimes easier to work with.

\begin{Lemma}\label{L:equiDef}
For each positive integer $m$, $L(m)$ is the minimum value of $L(D)$ over all covering designs $D=(X,\B)$ with $|X|=|\B|=m$.
\end{Lemma}

\begin{proof}
Lemma~\ref{L:symCover} implies that the quantity $\nu^*(m)$ defined in Lemma~\ref{L:LVsFracIndepSet}(b) can be equivalently defined as the maximum value of $\nu^*(H)$ over all intersecting hypergraphs $H$ with $m$ vertices and at most $m$ edges. Since edges may be given weight 0 in a fractional matching this is clearly also the maximum value of $\nu^*(H)$ over all intersecting hypergraphs $H$ with $m$ vertices and exactly $m$ edges. The result now follows from Lemma~\ref{L:LVsFracIndepSet}(b).
\end{proof}

We now discuss a relationship with another well-studied design theoretic parameter.
For positive integers $k$ and $v$ with $k \leq v$, the \emph{covering number $C(v,k,2)$} is the minimum number of blocks in a $k$-uniform covering design with $v$ points. In \cite{Mil}, Mills considers the quantity
\[\max\left\{\mfrac{v}{k}:k,v \in \mathbb{Z}, 1 \leq k \leq v \text{ and } C(v,k,2) \leq m\right\}\]
for given values of $m$, showing that this maximum does indeed exist and computing it exactly for positive integers $m \leq 13$. Our next lemma shows that $L(m)$ is in fact the reciprocal of this quantity.

\begin{Lemma}\label{L:LVsCovNumber}
For each positive integer $m$, it holds that
\begin{equation}\label{E:LVsCovNumber}
L(m) = \min\left\{\mfrac{k}{v}:k,v \in \mathbb{Z}, 1 \leq k \leq v \text{ and } C(v,k,2) \leq m\right\}.
\end{equation}
\end{Lemma}

\begin{proof}
Let $k$ and $v$ be integers such that $1 \leq k \leq v$ and $C(v,k,2) \leq m$. Suppose that $D=(X,\B)$ is a $k$-uniform covering design with $|\B|=m$ and that $w$ is the normalised weighting of $X$ such that $w(x)=\frac{1}{v}$ for each $x \in X$. Then
\[L(m) \leq L(D) \leq L(D,w) = \mfrac{k}{v}.\]
So $L(D) \leq \min\{\frac{k}{v}:k,v \in \mathbb{Z}, 1 \leq k \leq v \text{ and } C(v,k,2) \leq m\}$ by definition.

We now prove the opposite inequality. We do not retain any notation from the first part of the proof. By Lemma~\ref{L:LVsFracIndepSet}(b), there is a covering design $D=(X,\B)$ with $|\B|=m$ and a rational normalised weighting $w$  of $X$ such that $L(D,w) = L(m)$. Let $v$ be the least positive integer such that $v \times w(x)$ is an integer for each $x \in X$. Let $X'$ and $\B'$ be obtained from $X$ and $\B$ by replacing, in $X$ and in each $B \in \B$, each point $x$ with $v \times w(x)$ new points. Let $w'$ be the normalised weighting of $X'$ such that $w'(x)=\frac{1}{v}$ for each $x \in X'$. It is not difficult to see that $D'=(X',\B')$ is a covering design with $|X'|=v$ and $|\B'|=m$ and that $L(D',w')=L(D,w)$. Let $k$ be the size of a largest block in $\B'$ and let $D''=(X,\B'')$ be a $k$-uniform covering design where $\B''$ is obtained from $\B'$ by replacing each $B \in \B$ with a superset of $B$ of size $k$. Note that $C(v,k,2) \leq m$ by the existence of $D''$. Then
\[L(m)=L(D,w)=L(D',w')=L(D'',w')=\mfrac{k}{v}.\]
Thus $L(m) \geq \min\{\frac{k}{v}:k,v \in \mathbb{Z}, 1 \leq k \leq v \text{ and } C(v,k,2) \leq m\}$. This completes the proof.
\end{proof}

Via Lemma~\ref{L:LVsCovNumber}, the aforementioned computations of Mills in \cite{Mil} immediately give us the values of $L(m)$ listed in the following table. These values have also been computed independently by Kelly \cite{Kel} in the course of the research that led to \cite{HalKelTia}.

\begin{center}
\begin{tabular}{c|ccccccccccccc}
     $m$&  1 & 2 & 3 & 4 & 5 & 6 & 7 & 8 & 9 & 10 & 11 & 12 & 13\\\hline
     $L(m)$\rule{0cm}{4mm}& 1 & 1 & $\frac{2}{3}$ & $\frac{3}{5}$ & $\frac{5}{9}$ & $\frac{1}{2}$ & $\frac{3}{7}$ & $\frac{5}{12}$ & $\frac{2}{5}$ & $\frac{3}{8}$ & $\frac{5}{14}$ & $\frac{1}{3}$ & $\frac{4}{13}$
\end{tabular}
\end{center}

\section{The cases \texorpdfstring{$\bm{m=s^2+s}$}{m=s2+s} and \texorpdfstring{$\bm{m=s^2+s+1}$}{m=s2+s+1}} \label{S:projAff}

In \cite{HalKelTia} it was shown, for a positive integer $s$, that $L(s^2+s+1)=\frac{s+1}{s^2+s-1}$ if a projective plane of order $s$ exists. Also, $L(s^2+s)=\frac{1}{s}$ if an affine plane of order $s$ exists. In this section we show that the converse of the first of these statements holds and establish a partial converse for the second. Furthermore, we give a characterisation of the covering designs that can achieve these bounds.

The following two simple lemmas will be of use to us throughout this section and the next. We say that a weighting is \emph{positive} if it never assigns a weight of 0.

\begin{Lemma}\label{L:avWeightBound}
If $D=(X,\B)$ is a covering design and $w$ is a normalised weighting of $X$, then
\[|\B|\,L(D,w) \geq \sum_{x \in X}r_xw(x).\]
\end{Lemma}

\begin{proof}
This follows immediately from the observations that $\sum_{B \in \B}w(B) = \sum_{x \in X}r_xw(x)$ and that each of the $m$ blocks in $\B$ has weight at most $L(D,w)$.
\end{proof}

\begin{Lemma}\label{L:maxWeight}
Let $D=(X,\B)$ be a covering design with $|\B|=m$, and let $w$ be a normalised weighting of $X$ such that $L(D,w)<1$. For each $x \in X$ we have $r_x \geq 2$ and
\begin{equation}\label{E:vertWeightBound}
w(x) \leq \mfrac{r_x\,L(D,w)-1}{r_x-1}.
\end{equation}
Furthermore, if $w$ is positive, then we have equality in \eqref{E:vertWeightBound} only if $w(B)=L(D,w)$ for each $B \in \B_x$ and each $y \in X \setminus \{x\}$ is in exactly one block in $\B_x$.
\end{Lemma}

\begin{proof}
Let $x \in X$. We have $r_x \geq 2$ because otherwise, since $D$ is a covering, the unique block containing $x$ would be $X$ and hence $L(D,w)$ would be 1. Now the points in $X \setminus \{x\}$ have total weight $1-w(x)$ and each is in at least one block containing $x$. Hence,
\begin{equation}\label{E:blockWeightSum}
\sum_{B \in \B_x}w(B) \geq 1-w(x)+r_xw(x) = 1+(r_x-1)w(x).
\end{equation}
So, because $w(B) \leq L(D,w)$ for each $B \in \B_x$, we have
\begin{equation}\label{E:blockWeightBound}
L(D,w) \geq \mfrac{1}{r_x}\sum_{B \in \B_x}w(B) \geq \mfrac{1+(r_x-1)w(x)}{r_x}.
\end{equation}
Rearranging this gives \eqref{E:vertWeightBound}.

Furthermore, to have equality in \eqref{E:vertWeightBound}, we must have equality throughout \eqref{E:blockWeightBound} and hence in \eqref{E:blockWeightSum}. Equality in the first inequality in \eqref{E:blockWeightBound} implies that $w(B)=L(D,w)$ for each $B \in \B_x$. If $w$ is positive, equality in \eqref{E:blockWeightSum} implies that each $y \in X \setminus \{x\}$ is in exactly one block in $\B_x$.
\end{proof}

For a covering design $D=(X,\B)$ and a point $x$ in $X$, we say the covering design $D'=(X',\B')$ is \emph{obtained from $D$ by removing $x$} if $X'=X \setminus \{x\}$ and $\B'=\{B \setminus \{x\}: B \in \B\}$. If a covering design $D''$ is obtained from $D$ by iteratively applying this operation some number of times (including zero) we say it is \emph{obtained from $D$ by removing points}. A point $x$ in a covering design $D$ is said to be \emph{duplicated} if there is another point in $X$ that incident with exactly the same set of blocks. Recalling the definition of a simple hypergraph, a covering design has no duplicated points if and only if its dual is simple.

\begin{Remark}\label{R:removing}
For any covering design $D$ and normalised weighting $w$ of the points of $D$, we can obtain a design $D'$ with no duplicated points and a normalised positive weighting $w'$ of $D'$ such that $L(D',w')=L(D,w)$. This can be accomplished by first removing any points of $D$ with weight 0 and then iteratively removing duplicated points until no more remain, at each stage transferring any weight on the removed point to another point incident with the same set of blocks.
\end{Remark}

We will use the following special case of a result of F{\"u}redi \cite[Corollary~1]{Fur}.

\begin{Theorem}[{\cite[Corollary~1]{Fur}}]\label{T:Furedi}
Let $H=(X,\E)$ be a simple intersecting hypergraph and define $c=\max\{|E|-1:E \in \E\}$. Then $\nu^*(H) \leq c$ unless $H$ is a projective plane of order $c$.
\end{Theorem}

Combining this with Lemma~\ref{L:LVsFracIndepSet}, it is easy to show the following.

\begin{Lemma}\label{L:FurediConsequence}
Let $c$ be a positive integer and let $D$ be a covering design with no duplicated points such that each point occurs in at most $c+1$ blocks. If $L(D) < \frac{1}{c}$, then $D$ is a projective plane of order $c$.
\end{Lemma}

\begin{proof}
Suppose that $L(D) < \frac{1}{c}$. By our hypotheses, the dual $H$ of $D$ is a simple intersecting hypergraph in which each edge has size at most $c+1$. Also $\nu^*(H)>c$ by Lemma~\ref{L:LVsFracIndepSet}(a) because $L(D) < \frac{1}{c}$. So by Theorem~\ref{T:Furedi}, $H$ is a projective plane of order $c$. The dual of a finite projective plane is another finite projective plane of the same order, so $D$ is also a projective plane of order $c$.
\end{proof}

We can now prove a result which establishes Theorem~\ref{T:projAff}(a) and also characterises the covering designs that can achieve the bound.

\begin{Lemma}\label{L:projPlanesEtc} Let $s$ be a positive integer, let $D=(X,\B)$ be covering design with $s^2+s+1$ blocks and no duplicated points, and let $w$ be a positive normalised weighting of $X$ such that $L(D,w)=\frac{s+1}{s^2+s+1}$. Then $D$ is a projective plane of order $s$.
\end{Lemma}

\begin{proof}
Suppose that $D=(X,\B)$ is a covering design with $|\B|=s^2+s+1$ and no duplicated points, and that $w$ is a positive normalised weighting of the points in $X$ such that $L(D,w)=\frac{s+1}{s^2+s+1}$. So $L(D)=\frac{s+1}{s^2+s+1}$ by Theorem~\ref{T:HKTBound}. We must have $r_x \geq s+1$ for each $x \in X$ because otherwise, for some $x \in X$, we would have $r_x\, L(D,w) < 1$ and hence $w(x) < 0$  by Lemma~\ref{L:maxWeight}. So, since $L(D,w)=\frac{s+1}{s^2+s+1}$ and $w$ is positive, it must be the case that $r_x = s+1$ for each $x \in X$ by Lemma~\ref{L:avWeightBound}. Thus, by Lemma~\ref{L:FurediConsequence}, $D$ is a projective plane of order $s$ because $L(D,w)<\frac{1}{s}$.
\end{proof}

We now turn our attention to proving Theorem~\ref{T:projAff}(b). In this case we will make use of a result due to Bierbrauer, Marcugini and Pambianco \cite[Corollary~1]{BieMarPam} that extends Theorem~\ref{T:Furedi}.

\begin{Theorem}[{\cite[Corollary~1]{BieMarPam}}]\label{T:BieMarPam}
Let $c$ be a positive integer and let $H=(X,\E)$ be a simple intersecting hypergraph such that $\nu^*(H) = c$ and $|E| \leq c+1$ for each $E \in \E$. Then either $(X,\E)$ is an almost projective plane of order $c$ or there is a subset $\E'$ of $\E$ such that the dual of the hypergraph $(X,\E')$ is an affine plane of order $c$.
\end{Theorem}

\begin{Lemma}\label{L:BieMarPamConsequence}
Let $c$ be a positive integer and let $D$ be a covering design with no duplicated points such that each point occurs in at most $c+1$ blocks. If $L(D) = \frac{1}{c}$, then either $D$ is an almost projective plane of order $c$ or an affine plane of order $c$ can be obtained from $D$ by removing points.
\end{Lemma}

\begin{proof}
Suppose that $L(D) = \frac{1}{c}$. By our hypotheses, the dual $H$ of $D$ is a simple intersecting hypergraph in which each edge has size at most $c+1$. Also $\nu^*(H)=c$ by Lemma~\ref{L:LVsFracIndepSet}(a) because $L(D) = \frac{1}{c}$. So by Theorem~\ref{T:BieMarPam}, the dual of an affine plane of order $c$ can be obtained from $H$ by removing edges or $H$ is an almost projective plane of order $c$. In the former case we immediately have that an affine plane of order $c$ can be obtained from $D$ by removing points. In the latter case, it follows from our definition of a finite almost projective plane that its dual is another finite almost projective plane of the same order, so $D$ is also an almost projective plane or order $c$.
\end{proof}

The next lemma establishes Theorem~\ref{T:projAff}(b) and also characterises the covering designs that can achieve the bound. It also shows that $L(m) > \frac{1}{s}$ for positive integers $m$ and $s$ with $m < s^2+s$, which will be useful in our proof of Theorem~\ref{T:newBound} in Section~\ref{S:newCBound}.

\begin{Lemma}\label{L:affinePlanesEtc}
Let $s \geq 2$ be an integer and let $D=(X,\B)$ be a covering design such that $L(D) \leq \frac{1}{s}$. Then $|\B| \geq s^2+s$. Furthermore, if $|\B|=s^2+s$, $D$ has no duplicated points and $w$ is a positive normalised weighting of $X$ such that $L(D,w) = \frac{1}{s}$, then either an affine plane of order $s$ can be obtained from $D$ by removing points or $D$ is an almost projective plane of order $s$.
\end{Lemma}

\begin{proof}
Suppose $w$ is a normalised weighting of $X$ such that $L(D,w) \leq \frac{1}{s}$. Let $X^*={\{x \in X:w(x)>0\}}$. For each $x \in X^*$, we must have $r_x \geq s+1$ (otherwise we would have $w(x) \leq 0$, using Lemma~\ref{L:maxWeight} and $L(D,w) \leq \frac{1}{s}$). Thus, since $L(D,w) \leq \frac{1}{s}$, Lemma~\ref{L:avWeightBound} implies that
\begin{equation}\label{E:forcingAffine}
\mfrac{|\B|}{s} \geq |\B|\,L(D,w) \geq \sum_{x \in X}r_xw(x)=\sum_{x \in X^*}r_xw(x) \geq (s+1)\sum_{x \in X^*}w(x)=s+1.
\end{equation}
and it follows that $|\B| \geq s^2+s$.

Now further suppose that $|\B|=s^2+s$, $D$ has no duplicated points and $w$ is positive. Then $X^*=X$ and we also must have equality throughout \eqref{E:forcingAffine}. It follows that $L(D,w)=\frac{1}{s}$ and $r_x = s+1$ for each $x \in X$. So $L(D)=\frac{1}{s}$ by Theorem~\ref{T:HKTBound}. By Lemma~\ref{L:BieMarPamConsequence}, either an affine plane of order $s$ can be obtained from $D$ by removing points or $D$ is an almost projective plane of order $s$.
\end{proof}

It is now not difficult to prove Theorem~\ref{T:projAff} from Lemmas~\ref{L:projPlanesEtc} and \ref{L:affinePlanesEtc}.

\begin{proof}[\textup{\textbf{Proof of Theorem~\ref{T:projAff}}}]
We first prove part (a). If there exists a projective plane of order $s$, then $L(s^2+s+1)=\frac{s+1}{s^2+s+1}$ by Theorem~\ref{T:HKTBound}. Conversely, suppose that $L(s^2+s+1)=\frac{s+1}{s^2+s+1}$. Then there is a covering design $D=(X,\B)$ with $|\B|=s^2+s+1$ and a normalised weighting $w$ of $X$ such that $L(D,w)=\frac{s+1}{s^2+s+1}$. By Remark~\ref{R:removing}, by (if necessary) removing points from $D$, we can obtain a covering design $D'=(X',\B')$ with no duplicated points and a normalised positive weighting $w'$ of $D'$ such that $L(D',w')=\frac{s+1}{s^2+s+1}$. Then, by Lemma~\ref{L:projPlanesEtc}, $D'$ is a projective plane of order $s$. Furthermore, if $|X|=s^2+s+1$, then, noting that we also have $|X'|=s^2+s+1$ since $D'$ is a projective plane of order $s$, it must be the case that $X=X'$ and $D=D'$.

Now we prove part (b). If there exists an affine plane of order $s$, then $L(s^2+s)=\frac{1}{s}$ by Theorem~\ref{T:HKTBound}. Also, if there exists an almost projective plane of order $s$, then assigning each of its points weight $\frac{1}{s^2+s}$ shows that $L(s^2+s)=\frac{1}{s}$ in view of the lower bound from Theorem~\ref{T:HKTBound}. Conversely, suppose that $L(s^2+s)=\frac{1}{s}$. Then there is a covering design $D=(X,\B)$ with $|\B|=s^2+s$ and a normalised weighting $w$ of $X$ such that $L(D,w)=\frac{1}{s}$. As in the proof of (a), by (if necessary) removing points from $D$, we can obtain a covering design $D'=(X',\B')$ with no duplicated points and a normalised positive weighting $w'$ of $D'$ such that $L(D',w')=\frac{1}{s}$. Then, by Lemma~\ref{L:affinePlanesEtc}, either an affine plane of order $s$ can be obtained from $D'$ by removing points or $D'$ is an almost projective plane of order $s$. So either an affine plane of order $s$ or an almost projective plane of order $s$ exists. Furthermore, if $|X|=s^2$, then, noting that we also have $|X'|=s^2$ since $D'$ is an affine plane of order $s$, it must be the case that $X=X'$ and $D=D'$.
\end{proof}

We noted in the introduction that almost projective planes of orders $2$ and $3$ exist but no example having order greater than 3 has been found (see \cite{BieMarPam,BloJunSch}). The following example uses an almost projective plane of order $3$. It emphasises that there do exist covering designs $D$ that have $s^2+s$ blocks and $L(D)=\frac{1}{s}$, but which are not affine planes.

\begin{Example}
\label{X:Z12}
Consider the covering design $D=(X,\B)$ where $X=\mathbb{Z}_{12}$ and $\B=\{\{i,1+i,4+i,6+i\}:i \in \mathbb{Z}_{12}\}$. Then $|X|=|\B|=12$ and $r_x=4$ for each $x \in X$, so $D$ is an almost projective plane of order 3. Because $L(m)=\frac{1}{3}$ by Theorem~\ref{T:HKTBound}, assigning each point weight $\frac{1}{12}$ shows that $L(D)=\frac{1}{3}$.
\end{Example}

Bruck and Ryser \cite{BruRys} famously showed that projective planes of certain orders cannot exist.

\begin{Theorem}[{\cite{BruRys}}]\label{T:PPnonexistence}
If a projective plane of order $s$ with $s \equiv 1,2 \mod{4}$ exists, then $s$ is a sum of two squares.
\end{Theorem}

 Blokhuis, Jungnickel and Schmidt \cite[Proposition 1.2]{BloJunSch} used results of Bose and Connor \cite{BosCon} to prove a similar result for almost projective planes.

\begin{Theorem}[{\cite{BloJunSch}}]\label{T:APPnonexistence}
Suppose an almost projective plane of order $s$ exists and let $m = \binom{s+1}{2}$. When $s \equiv 0,3 \mod{4}$, we have that $s+1$ is a square and, further, if $m \equiv 2 \mod{4}$, then $s-1$ is a sum of two squares. When $s \equiv 1,2 \mod{4}$, then $s-1$ is a square, and there are integers $x$, $y$ and $z$ such that
\[(s+1)x^2 + (-1)^{m(m-1)/2}\, 2y^2 = z^2.\]
\end{Theorem}

For a given positive integer $s$, the consequences of Theorem~\ref{T:projAff} vary according to the state of knowledge about whether a projective plane of order $s$ exists.
\begin{enumerate}
    \item
If $s$ is a power of a prime, a projective plane of order $s$ exists and so we have $L(s^2+s)=\frac{1}{s}$ and $L(s^2+s+1)=\frac{s+1}{s^2+s+1}$.
    \item
If it is not known whether a projective plane of order $s$ exists, we cannot determine whether either of $L(s^2+s)=\frac{1}{s}$ or $L(s^2+s+1)=\frac{s+1}{s^2+s+1}$ holds since no almost projective plane of order greater than 3 is known.
    \item
If Theorem~\ref{T:PPnonexistence} rules out the existence of a projective plane of order $s$, then Theorem~\ref{T:APPnonexistence} also establishes the nonexistence of an almost projective plane of order $s$ (any order allowed to exist by Theorem~\ref{T:APPnonexistence} is $i^2+1$ or $i^2-1$ for some $i$ and hence either is a sum of two squares or is congruent to 0 or 3 modulo 4). Thus, we have that $L(s^2+s) \neq \frac{1}{s}$ and $L(s^2+s+1) \neq \frac{s+1}{s^2+s+1}$ in these situations.
    \item
The nonexistence of a projective plane of order 10 has been established using heavy computation \cite{LamThiSwi}. However an almost projective plane of order 10 is not ruled out by Theorem~\ref{T:APPnonexistence} (note that $(x,y,z)=(1,1,3)$ is a solution to $11x^2-2y^2=z^2$). So, if $s=10$, we have $L(s^2+s+1) \neq \frac{s+1}{s^2+s+1}$ but we do not know whether $L(s^2+s) = \frac{1}{s}$ or not.
\end{enumerate}

\section{An improved bound on the data limit}\label{S:newCBound}

Our main goal in this section is to prove Theorem~\ref{T:newBound}. However, we first note that Theorem~\ref{T:HKTBound} is already enough to give the asymptotic behaviour of the data limit.

\begin{Corollary}
As $m \rightarrow \infty$, we have $L(m) = (1+o(1))m^{-1/2}$.
\end{Corollary}

\begin{proof}
Let $m$ be a positive integer. Since $ \sqrt{m}-1 \leq \lfloor \sqrt{m} \rfloor \leq \sqrt{m}$, it follows from Theorem~\ref{T:HKTBound} that $L(m) \geq m^{-1/2}$. We will complete the proof by showing that  $L(m) \leq (1+o(1))m^{-1/2}$. Let $q=\lfloor \sqrt{m} - \frac{1}{2} \rfloor$. Provided that $m$ is sufficiently large, by the main result of \cite{BakHarPin}, there is a prime $p$ such that $q-q^{0.525} \leq p \leq q$. So there is an affine plane of order $p$ and $L(p^2+p)=\frac{1}{p}$ by Theorem~\ref{T:HKTBound}. Thus, because $p \leq q \leq \sqrt{m} - \frac{1}{2}$, we have $p^2+p \leq m$ and hence $L(m) \leq L(p^2+p)=\frac{1}{p}$. Now $q \geq \sqrt{m} - \frac{3}{2}$ and hence
\[p \geq \sqrt{m} - \tfrac{3}{2} - \left(\sqrt{m} - \tfrac{3}{2}\right)^{0.525} \geq \sqrt{m} - m^{0.2625} - \tfrac{3}{2}.\]
Thus
\[L(m) \leq \mfrac{1}{p} \leq \mfrac{1}{\sqrt{m}}+O\left(m^{-0.7375}\right). \qedhere\]
\end{proof}

We now turn our attention to proving Theorem~\ref{T:newBound}. Throughout this section we define $F:[2,\infty) \rightarrow \mathbb{R}$ to be the function given by
\[F(m) = \mfrac{s^2 + (2m-1)s - 1  + \sqrt{(s^2-s-1)^2 + 4m(s-1)(s^2 + s + 1 - m)}}{2 m (s^2+s-1)}\]
where $s$ is the unique positive integer such that $s^2-s+1< m \leq s^2+s+1$. The right hand side of \eqref{E:newBound} is exactly $F(m)$. Note that $F(m)=1$ for $m \in \{2,3\}$. We establish some other basic properties of $F$.

\begin{Lemma}\label{L:boundFacts}\phantom{a}
\begin{itemize}
    \item[\textup{(a)}]
For each positive integer $s$, $F(s^2+s)=\frac{1}{s}$ and $F(s^2+s+1)=\frac{s+1}{s^2+s+1}$.
    \item[\textup{(b)}]
$F$ is continuous and monotonically decreasing on $[3,\infty)$.
    \item[\textup{(c)}]
For each integer $m \geq 3$, $F(m)$ is at least the bound of Theorem~\ref{T:HKTBound}, with equality if and only if $m \in \{s^2+s,s^2+s+1\}$ for some positive integer $s$.
\end{itemize}
\end{Lemma}

\begin{proof}
Part (a) follows by making the appropriate substitutions and simplifying. For each integer $s \geq 2$, $F$ is  continuous on the interval $(s^2-s+1,s^2+s+1)$ with derivative
\[\mfrac{-2m(s^3-1)-(s^2-s-1)\bigl(s^2-s-1 + \sqrt{(s^2-s-1)^2 + 4m(s-1)(s^2+s+1-m)}\,\bigr)}
{2m(s^2+s+1)\sqrt{(s^2-s-1)^2+4m(s-1)(s^2+s+1-m)}}\]
which is negative. Thus, (b) follows by observing that, for each positive integer $s$, $F$ is also continuous at $s^2+s+1$. For (c), let $m \geq 3$ be an integer and let $s$ be the positive integer such that $s^2-s+1 < m \leq s^2+s+1$. The claims of equality follow from (a). If $s^2 \leq m < s^2+s$, then the bound of Theorem~\ref{T:HKTBound} is $\frac{1}{s}$ and $F(m) > F(s^2+s)=\frac{1}{s}$ using (a) and (b). If $s^2-s+1 < m < s^2$, then the bound of Theorem~\ref{T:HKTBound} is $\frac{s}{m}$ and simplification shows that $F(m)>\frac{s}{m}$ if and only if the radical in $F(m)$ is greater than $s(2s^2-2m+s-1)+1$. Squaring both sides of this inequality, simplifying and cancelling a common factor of $4(s^2+s-1)$ we see that in turn this holds if and only if
\begin{equation}\label{E:comparisonQuad}
-m^2 + (2s^2+1)m-4s(s^3+1) > 0.
\end{equation}
The left hand side of \eqref{E:comparisonQuad} is a concave quadratic function of $m$ with roots at $s^2-s+1$ and $s^2+s$. Thus it is positive for all $m$ satisfying $s^2-s+1 < m < s^2$ and we have the desired result.
\end{proof}

Figure~\ref{F:plot} gives a visual comparison of the bounds of Theorems~\ref{T:HKTBound} and \ref{T:newBound} for small values of $m$.

\begin{figure}[htb]
\begin{center}
\includegraphics[width=0.7\textwidth]{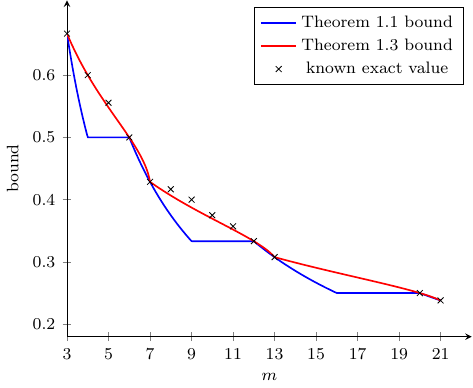}
\end{center}
\caption{A comparison of the bounds of Theorems~\ref{T:HKTBound} and \ref{T:newBound} and, where applicable, known exact values for $3 \leq m \leq 21$. For visual clarity we treat each of the two bounds as a function of a real variable $m$. Note that neither axis starts at 0.}\label{F:plot}
\end{figure}

We prove the main part of Theorem~\ref{T:newBound} in the following lemma.

\begin{Lemma}\label{L:newBound}
For each integer $m \geq 2$ we have $L(m) \geq F(m)$.
\end{Lemma}

\begin{proof}
Let $s$ be the positive integer such that $s^2-s+1< m \leq s^2+s+1$. For brevity, let $\ell=L(m)$. By Lemma~\ref{L:equiDef} there is a covering design $D=(X,\B)$ with $|X|=|\B|=m$ and a normalised weighting of $X$ such that $L(D,w)=\ell$. If $m = 2$, then $\ell=F(m)=1$. If $m=3$, then $\ell=F(m)=\frac{2}{3}$ by Theorem~\ref{T:HKTBound}. So we may assume that $m \geq 4$ and hence that $s \geq 2$. If $m \in \{s^2+s,s^2+s+1\}$, then the result follows by Theorem~\ref{T:HKTBound} and Lemma~\ref{L:boundFacts}(a). So we may assume that $m < s^2+s$ and hence, by the contrapositive of the first part of Lemma~\ref{L:affinePlanesEtc}, that $\ell>\frac{1}{s}$. By Lemma~\ref{L:boundFacts}(a) and (b) we have that $F(m) < F(s^2-s) =\frac{1}{s-1}$. So we may suppose that $\ell<\frac{1}{s-1}$ for otherwise we are done. Thus we have $r_x \geq s$ for each $x \in X$, for otherwise we would have a contradiction to $w(x) \geq 0$ using Lemma~\ref{L:maxWeight}.

For each $i \in \{s,\ldots,m\}$, let $X_i=\{x \in X:r_x=i\}$ and $y_i=\sum_{x \in X_i}w(x)$. Note that $\{X_s,\ldots,X_m\}$ is a partition of $X$. Since $w$ is normalised we have
\begin{equation}\label{E:unitWeight}
    \sum_{i=s}^m y_i = 1.
\end{equation}
By Lemma~\ref{L:maxWeight}, for each $i \in \{s,\ldots,m\}$, we have $y_i \leq \frac{i\ell-1}{i-1}|X_i|$ and hence $\frac{i-1}{i\ell-1}\,y_i \leq |X_i|$. Thus, because $\sum_{i=s}^m|X_i|=m$, we have
\begin{equation}\label{E:numVerts}
    \sum_{i=s}^m \mfrac{i-1}{i\ell-1}\,y_i \leq m.
\end{equation}
Furthermore, by Lemma~\ref{L:avWeightBound},
\begin{equation}\label{E:totalWeight}
    m\ell \geq \sum_{i=s}^m iy_i.
\end{equation}

We will now effectively bound $\sum_{i=s}^m iy_i$ below by the optimal value of the linear program that minimises $\sum_{i=s}^m iy_i$ subject to the constraints \eqref{E:unitWeight} and \eqref{E:numVerts}. Take $\alpha$ times equality \eqref{E:unitWeight} and subtract $\beta$ times inequality \eqref{E:numVerts} where $\alpha=\frac{(s^2-s-1)\ell+1}{1-\ell}$ and $\beta=\frac{(s\ell-1)(s\ell+\ell-1)}{1-\ell}$ are positive. After simplification, this yields
\begin{equation}\label{E:linComb}
\sum_{i=s}^m (i-c_i)y_i \geq \mfrac{(s^2+2ms+m-s-1)\ell-m(s^2+s)\ell^2-m+1}{1-\ell}\text{~~~where~~~} c_i=\mfrac{(i-s)(i-s-1)\ell}{i\ell-1}.
\end{equation}
Note that our choice of $\alpha$ and $\beta$ was the unique one under which the resulting inequality would have $s$ as the coefficient of $y_s$ and $s+1$ as the coefficient of $y_{s+1}$. Clearly, $c_i \geq 0$ for each $i \in \{s,\ldots,m\}$ and so from \eqref{E:totalWeight} and \eqref{E:linComb} we can deduce that $m\ell$ is at least the right hand side of \eqref{E:linComb}. The resulting inequality simplifies to
\begin{equation}\label{E:ellQuadratic}
m(s^2+s-1)\ell^2-(s^2+2ms-s-1)\ell + m-1 \geq 0
\end{equation}
which we can view as a quadratic function of $\ell$ with two real roots. We saw earlier that $\ell > \frac{1}{s}$ and we now observe that substituting $\ell=\frac{1}{s}$ into \eqref{E:ellQuadratic}
yields
\[\mfrac{m(s^2+s-1)}{s^2} - \mfrac{(s^2 +2sm -s-1)}{s} +m-1 = \mfrac{(s-1)(m-s^2-s)}{s^2}\]
which is nonpositive since $m \leq s^2+s$. Thus, because $\ell > \frac{1}{s}$, we have that $\ell$ must be at least the larger of two roots of \eqref{E:ellQuadratic}. This larger root is exactly $F(m)$.
\end{proof}

We now proceed to prove the rest of Theorem~\ref{T:newBound} by re-examining the proof of Lemma~\ref{L:newBound}, analysing how equality in the bound could arise. A \emph{near pencil of order $m$} is a linear space with $m$ points and $m$ blocks such that one block has size $m-1$ and every other block has size $2$.

\begin{Lemma}\label{L:newBoundEquality}
For each integer $m \geq 2$, if $L(m)=F(m)$, then $m=4$ or $m \in \{s^2+s,s^2+s+1\}$ for some positive integer $s$.
\end{Lemma}

\begin{proof}
For brevity, let $\ell=L(m)$. Suppose that $\ell=F(m)$ and that $m \notin \{s^2+s,s^2+s+1\}$ for any positive integer $s$. We must show that $m=4$. By Lemma~\ref{L:equiDef} there is a covering design $D=(X,\B)$ with $|X|=|\B|=m$ and a normalised weighting $w$ of $X$ such that $L(D,w)=\ell$. Let $s$ be the positive integer such that $s^2-s+1 < m \leq s^2+s+1$ and note $m < s^2+s$ by our supposition. Thus $\ell = F(m) > \frac{1}{s}$ by Lemma~\ref{L:boundFacts}(a) and (b). For each $i \in \{s,\ldots,m\}$, let $X_i=\{x \in X:r_x=i\}$ and $y_i=\sum_{x \in X_i}w(x)$.

Following the proof of Lemma~\ref{L:newBound}, we have $r_x \geq s$ for each $x \in X$. Further, because $\ell=F(m)$, we must have equality in \eqref{E:ellQuadratic} and hence we must have equality in both \eqref{E:totalWeight} and \eqref{E:linComb}, and the right side of \eqref{E:totalWeight} must equal the left side of \eqref{E:linComb}. Note that, for $c_i$ as defined in \eqref{E:linComb}, we have $c_i = 0$ for $i \in \{s,s+1\}$ and $c_i>0$ for $i \in \{s+2,\ldots,m\}$. Thus equality between the right side of \eqref{E:totalWeight} and the left side of \eqref{E:linComb} gives $\sum_{i=s}^m (i-c_i)y_i=\sum_{i=s}^m iy_i$ which implies that $y_i=0$ for $i \in \{s+2,\ldots,m\}$. Since \eqref{E:linComb} is obtained from an equality by subtracting a positive multiple of inequality \eqref{E:numVerts}, equality in \eqref{E:linComb} implies equality in \eqref{E:numVerts}. This in turn implies that $y_i=\frac{i\ell-1}{i-1}|X_i|$ for each $i \in \{s,\ldots,m\}$ (note that $\frac{i\ell-1}{i-1}$ is positive since $\ell > \frac{1}{s}$). Hence $|X_i|=0$ for each $i \in \{s+2,\ldots,m\}$ and $w(x)=\frac{i\ell-1}{i-1}$ for all $x \in X_i$ and $i \in \{s,s+1\}$. It follows that $w$ is positive and we have equality in \eqref{E:vertWeightBound} for each $x \in X$. So, by applying Lemma~\ref{L:maxWeight} to each $x \in X$, we see that any two points in $X$ occur together in exactly one block and so $(X,\B)$ is a linear space with $|X|=|\B|=m$. The De Bruijn–Erd\H{o}s theorem in incidence geometry \cite{DebErd} states that such a space must be either a finite projective plane or a near pencil.

Now $D$ cannot be a finite projective plane because $m \neq s^2+s+1$, and hence $D$ is a near-pencil. So $r_y=m-1$ for some $y \in X$ and $r_x=2$ for each $x \in X \setminus \{y\}$. However $r_x \in \{s,s+1\}$ for all $x \in X$ because we have seen that $|X_i|=0$ for each $i \in \{s+2,\ldots,m\}$. So it must be the case that $s=2$ and $m=4$.
\end{proof}

Recall that we saw in Example~\ref{X:order4} that we do indeed have $L(4)=F(4)=\frac{3}{5}$ and that this is achieved by an appropriately weighted near pencil of order 4.

\begin{proof}[\textup{\textbf{Proof of Theorem~\ref{T:newBound}}}]
This follows from Lemmas~\ref{L:newBound} and \ref{L:newBoundEquality}.
\end{proof}

\section{Data limits for certain classes of covering designs}\label{S:designClasses}

In addition to finding bounds on $L(m)$, we believe it is also of interest to find bounds on $L(D)$ for various classes of covering designs. We investigate such bounds in this section. Our first result gives a general method for finding a lower bound on $L(D)$ by choosing a certain weighting of the blocks of $D$.

\begin{Lemma}\label{L:fracTransBound}
Let $D=(X,\B)$ be a covering design and $h$ be a weighting of $\B$ such that $\sum_{B \in \B_x}h(B) \geq 1$ for each $x \in X$. Then
\[L(D) \geq \frac{1}{\sum_{B \in \B}h(B)}.\]
\end{Lemma}

\begin{proof}
Let $w$ be a weighting of $X$ such that $L(D,w)=L(D)$. We have
\[L(D)\sum_{B \in \B}h(B) \geq \sum_{B \in \B}\biggl(h(B)\sum_{x \in B}w(x)\biggr) \geq  \sum_{x \in X}w(x) = 1\]
where the first inequality follows because $\sum_{x \in B}w(x) \leq L(D)$ for each $B \in \B$ and the second follows because $\sum_{B \in B_x}h(B) \geq 1$ for each $x \in X$.
\end{proof}

We note that the weighting $h$ in Lemma~\ref{L:fracTransBound} induces a \emph{fractional transversal of value $\sum_{B \in \B}h(B)$} in the hypergraph $H$ which is the dual of $D$ (see \cite{Fur} for definitions and context). The minimum value of a transversal in $H$ is $\nu^*(H)$ because the linear program for finding this minimum is the dual of the linear program for finding the maximum size of a fractional matching in $H$. Combining this observation with Lemma~\ref{L:LVsFracIndepSet} gives an alternative proof of Lemma~\ref{L:fracTransBound}. We observe the following useful corollary of Lemma~\ref{L:fracTransBound}.

\begin{Corollary}\label{C:fracTransCor}
Let $D=(X,\B)$ be a covering design. If there is a submultiset $\B'$ of $\B$ and a positive integer $t$ such that each point in $X$ occurs in at least $t$ blocks in $\B'$, then $L(D) \geq t/|\B'|$.
\end{Corollary}

\begin{proof}
Apply Lemma~\ref{L:fracTransBound} with $h(B)=\frac{1}{t}$ if $B \in \B'$ and $h(B)=0$ otherwise.
\end{proof}

We can use Corollary~\ref{C:fracTransCor} to exactly determine the data limit for coverings that are both uniform and regular and to give some general bounds relating to the replication numbers and block sizes of a covering design.

\begin{Lemma}\label{L:rkBounds}
If $D=(X,\B)$ is a covering design, then
\[\frac{\min\{r_x:x\in X\}}{|\B|} \leq L(D) \leq \frac{\max\{|B|:B \in \B\}}{|X|}.\]
In particular, $L(D)=\frac{r}{|X|}$ if $r_x=r$ for all $x \in X$ and $|B|=|B'|$ for all $B,B' \in \B$.
\end{Lemma}

\begin{proof}
Let $v=|X|$, $m=|\B|$, $r_{\min}=\min\{r_x:x\in X\}$ and $k_{\max}=\max\{|B|:B \in \B\}$. Applying Corollary~\ref{C:fracTransCor} with $\B'=\B$ shows that $L(m) \geq \frac{r_{\min}}{m}$. Furthermore, $L(D,w) = \frac{k_{\max}}{v}$ where $w$ is the normalised weighting of $X$ given by $w(x)=\frac{1}{v}$ for each $x \in X$. This proves the first part of the theorem.

Now suppose that, for some integers $r$ and $k$, $r_x=r$ for all $x \in X$ and $|B|=k$ for all $B \in \B$. Then $\frac{r}{m} \leq L(D) \leq \frac{k}{v}$ from the first part and $rv=km$ by counting point-block incidences. So $L(D)=\frac{r}{m}$.
\end{proof}

Note that Lemma~\ref{L:rkBounds} generalises results from \cite{HalKelTia} concerning balanced incomplete block designs. A \emph{balanced incomplete block design} is a uniform linear space. Such designs necessarily have exactly $r$ blocks containing each point for some fixed $r$ and so the second part of Lemma~\ref{L:rkBounds} applies to them. We now establish another lower bound on $L(D)$ based on the replication numbers of the points of a covering design $D$.

\begin{Lemma}\label{L:repSeqBound}
If $D = (X,\B)$ is a covering design such that $L(D)<1$ and $\sigma = \sum_{x \in X} \frac{1}{r_x-1}$, then
\[L(D) \geq \mfrac{1+\sigma}{|X|+\sigma}.\]
\end{Lemma}

\begin{proof}
Note that $r_x \geq 2$ for each $x \in X$ because otherwise $X \in \B$ and $L(D)=1$. Let $w$ be a normalised weighting of $X$ such that $L(D,w)=L(D)$. Applying Lemma~\ref{L:maxWeight} to each point in $X$ and summing the resulting inequalities we see that
\[1 \leq \sum_{x \in X}\mfrac{r_xL(D)-1}{r_x-1}=L(D)\biggl(\,\sum_{x \in X}\mfrac{r_x}{r_x-1}\biggr)-\sigma=L(D)\bigl(|X|+\sigma\bigr)-\sigma.\]
Solving this inequality for $L(D)$ completes the proof.
\end{proof}

We give a small example to illustrate the application of some of the lemmas we have just proven.

\begin{Example}
Let $D=(X,\B)$ be the covering design with $X=\{1,2,3,4,5\}$ and
\begin{equation}\label{E:exBlockList}
\B=\bigl\{\{1,2\},\{1,3\},\{1,4,5\},\{2,3,4\},\{2,3,5\}\bigr\}.
\end{equation}
We have $r_1=r_2=r_3=3$ and $r_4=r_5=2$, and hence Lemma~\ref{L:rkBounds} gives $L(D) \geq \frac{2}{5}$. Lemma~\ref{L:repSeqBound} improves on this by giving $L(D) \geq \frac{9}{17}$. This can be improved still further by applying Lemma~\ref{L:fracTransBound} with block weights $\frac{1}{5},\frac{1}{5},\frac{3}{5},\frac{2}{5},\frac{2}{5}$ respectively
according to the order the blocks are listed in \eqref{E:exBlockList}. This weighting makes the sum of the block weights on each point exactly 1, and hence Lemma~\ref{L:fracTransBound} implies that $L(D) \geq \frac{5}{9}$. In fact, we have $L(D)=\frac{5}{9}$ because $L(D,w)=\frac{5}{9}$ for the weighting of $X$ given by $w(1)=\frac{1}{3}$, $w(2)=w(3)=\frac{2}{9}$ and $w(4)=w(5)=\frac{1}{9}$.
\end{Example}

As we just saw, Lemma~\ref{L:rkBounds} and Lemma~\ref{L:repSeqBound} give bounds based on natural parameters of the covering design. Lemma~\ref{L:fracTransBound} potentially yields stronger bounds, but this is reliant on finding a suitable choice of weighting for the blocks.

We noted that data limits for balanced incomplete blocks designs were considered in \cite{HalKelTia} and are covered by Lemma~\ref{L:rkBounds}. We will conclude this section by considering data limits for two other classes of covering design: transversal designs and finite projective Hjelmslev planes. Both of these classes are closely related to the class of finite projective planes.

A \emph{$(k,n)$-transversal design} is a linear space $(X,\B)$ such that $|X| = kn$ and $\B$ consists of
\begin{itemize}[itemsep=0mm]
\item $k$ blocks of size $n$ that partition $X$; and
\item $n^2$ other blocks of size $k$.
\end{itemize}
Note that a projective plane of order $n$ can be obtained from an $(n+1,n)$-transversal design by adding a fixed new point to each of the blocks of size $n$. In general a $(k,n)$-transversal design exists if and only if there are $k-2$ mutually orthogonal latin squares of order $n$ (see \cite[\S6.6]{Sti}, for example). The results we have proved so far in this section make it easy to determine the data limits of transversal deigns.

\begin{Theorem}\label{T:transversal}
If $D$ is a $(k,n)$-transversal design with $k \leq n$, then $L(D) = \frac{1}{k}$.
\end{Theorem}

\begin{proof}
Note that $D$ is a covering design with $kn$ points and blocks of sizes $k$ and $n$. Hence $L(D) \leq \frac{n}{kn} = \frac{1}{k}$ by Lemma~\ref{L:rkBounds}. For the lower bound, we apply Corollary~\ref{C:fracTransCor} with $\B'$ taken to be a set of $k$ blocks of size $n$ in $\B$ that partition $X$. This shows that $L(D) \geq \frac{1}{k}$.
\end{proof}

A \emph{$(t,q)$-projective Hjelmslev  plane} can be defined as an intersecting covering design $(X,\B)$ such that there exists a projective plane $(P,\mathcal{L})$ of order $q$, a partition $\{X_p:p \in P\}$ of $X$ and a partition $\{\B_L:L \in \mathcal{L}\}$ of $\B$ with the following three properties.\pagebreak
\begin{itemize}[itemsep=0mm]
\item For all $L \in \mathcal{L}$ and $p \in P$, we have $|L \cap X_p|=t$ if $p \in L$ and $|L \cap X_p|=0$ otherwise.
\item Any two points in $X$ that occur together in more than one block are in the same class in $\{X_p:p \in P\}$.
\item Any two blocks in $\B$ that intersect in more than one point are in the same class in $\{\B_L:L \in \mathcal{L}\}$.
\end{itemize}
It is known that for any $(t,q)$-projective Hjelmslev plane, we have $|X|=|\B|=t^2(q^2+q+1)$, $|B|=t(q+1)$ for each $B \in \B$ and $r_x=t(q+1)$ for each $x \in X$ (see \cite{Dra}, for example). A $(1,q)$-projective Hjelmslev plane is simply a projective plane of order $q$. From Lemma~\ref{L:rkBounds}, we immediately have the following.

\begin{Theorem} \label{T:Hjelmslev}
If $D$ is a $(t,q)$-projective Hjelmslev plane, then $L(D) = \frac{q+1}{t(q^2+q+1)}$.
\end{Theorem}

It is not fully known for which values of $t$ and $q$ there exists a $(t,q)$-projective Hjelmslev plane, but it is known that a $(q,q)$-projective Hjelmslev plane exists whenever a projective plane of order $q$ exists \cite{HalRao,HanVan}. We briefly consider the example of a $(2,2)$-projective Hjelmslev plane.

\begin{Example}
From the immediately preceding discussion, there exists a \emph{$(2,2)$-projective Hjelmslev plane}. Such a plane is a covering design $D=(X,\B)$ such that $|\B|=|X|=28$ and $|B|=6$ for each $B \in \B$. We have $L(D) = \frac{3}{14}$ by Theorem~\ref{T:Hjelmslev}. This establishes that $L(28) \leq \frac{3}{14} \approx 0.2143$. Theorem~\ref{T:newBound} establishes $L(28) \gtrapprox 0.2095$, improving on Theorem~\ref{T:HKTBound} which gives $L(28) \geq \frac{1}{5}= 0.2$.
\end{Example}

\section{Conclusion}\label{S:conc}

Our knowledge of $L(b)$ when $b$ is the number of blocks in a finite projective or affine plane, together with the fact that $L(m)$ is nondecreasing in $m$, gives us naive lower and upper bounds on $L(m)$ for all $m$. Theorem~\ref{T:newBound} provides a substantial improvement on these naive lower bounds. However, we have no general improvement on the naive upper bound given by $L(m) \leq \frac{s+1}{s^2+s+1}$ where $s$ is the greatest prime power such that $s^2+s+1 \leq m$. Of course, known upper bounds on covering numbers provide upper bounds on $L(m)$ via Lemma~\ref{L:LVsCovNumber}, and we have discussed how balanced incomplete block designs, transversal designs, and finite projective Hjelmslev planes can provide upper bounds. However, we believe the lack of more systematic upper bounds still warrants attention.

Results on the existence of almost projective planes would impact our results in Theorem~\ref{T:projAff} and would be of independent interest. For our purposes, we could omit the mention of almost projective planes from Theorem~\ref{T:projAff}(b) if it were shown that the orders for which almost projective planes exist formed a subset of those for which projective planes exist. Given that no almost projective plane of order greater than 3 is known, this may well be true. Proving it may not be a simple matter, however.

\subsection*{Acknowledgements}

We thank Marco Buratti for helpful discussions. D. Horsley's research is supported by the Australian Research Council grant DP220102212. D.R. Stinson's research is supported by NSERC discovery grant RGPIN-03882.

\end{document}